\documentclass[12pt,reqno]{amsart}

\usepackage[latin1]{inputenc}
\usepackage{amsmath}
\usepackage{amsfonts}
\usepackage{amssymb}
\usepackage{graphics}
\usepackage{enumerate}
\usepackage[dvipsnames]{xcolor}
\usepackage{amssymb,amsmath,amsthm,amscd,latexsym,verbatim,graphicx,amsfonts}

\usepackage{amscd}
\usepackage{amsmath}
\usepackage{amssymb}
\usepackage{amsthm}
\usepackage{latexsym}
\usepackage{verbatim}

\theoremstyle{plain}
\newtheorem{theorem}{Theorem}[section]

\newtheorem{lemma}[theorem]{Lemma}
\newtheorem{prop}[theorem]{Proposition}
\theoremstyle{definition}

\theoremstyle{remark}

\newcommand{\bbR}{\mathbb{R}}
\newcommand{\bbC}{\mathbb{C}}

\newcommand{\bbN}{\mathbb{N}}

\newcommand{\eitheta}{e^{i\theta}}

\title[]{Bounded Connected Components of Polynomial Lemniscates}
\author[]{Adam Kraus $\&$ Brian Simanek}
\date{}

\begin{document}
\maketitle

\begin{abstract}
We consider families of polynomial lemniscates in the complex plane and determine if they bound a Jordan domain.  This allows us to find examples of regions for which we can calculate the projection of $\bar{z}$ to the Bergman space of the bounded region.  Such knowledge has applications to the calculation of torsional rigidity.
\end{abstract}

\vspace{4mm}

\footnotesize\noindent\textbf{Keywords:} Bergman Projection, Polynomial Lemniscates, Torsional Rigidity

\vspace{2mm}

\noindent\textbf{Mathematics Subject Classification:} Primary 30E10; Secondary 74P10

\vspace{2mm}

\normalsize

\section{Introduction}\label{Intro}

Consider a bounded Jordan domain $\Omega$ in the complex plane $\bbC$.  We define the Bergman space $A^2(\Omega)$ to be the space of functions that are holomorphic in $\Omega$ and which have finite $L^2$-norm with respect to area measure on $\Omega$.  The \textit{Bergman Analytic Content} of $\Omega$ (denoted $\sigma(\Omega)$) is the distance from $\bar{z}$ to $A^2(\Omega)$ in the space $L^2(\Omega,dA)$.  Our motivation for considering this quantity is a result from \cite{FL17}, which states that if $\Omega$ is simply connected, then $\sigma(\Omega)^2$ is equal to the torsional rigidity of $\Omega$ (see \cite{Pol} for a definition and discussion of torsional rigidity).

One way to go about calculating $\sigma(\Omega)$ is to find the projection of $\bar{z}$ to $A^2(\Omega)$ in $L^2(\Omega)$.  A result of Fleeman and Khavinson shows that the projection of $\bar{z}$ to $A^2(\Omega)$ is $F'(z)$, where $\mbox{Re}[F(z)]=|z|^2/2$ on the boundary of $\Omega$ (see \cite{FK}).  Thus, one can find examples of regions $\Omega$ where one can calculate $\sigma(\Omega)$ exactly by considering regions whose boundary is part of the set $\tilde{\Gamma}_F:=\{z:\mbox{Re}[F(z)]=|z|^2/2\}$, as long as $F$ is holomorphic on $\Omega$.  If $F$ is a polynomial, then this last condition is trivially satisfied, so one need only find examples where the set $\tilde{\Gamma}_F$ bounds a Jordan domain.  

This task was taken up in \cite{FK}, where several examples are considered.  Since the Bergman projection of $\bar{z}$ is $F'(z)$, not $F(z)$, there is no loss of generality by instead considering the set
\[
\Gamma_F:=\{z:\mbox{Re}[F(z)+1]=|z|^2\}=\{z:\mbox{Re}[F(z)]=|z|^2-1\},
\]
which is slightly easier for some computations. It was asked in \cite{FK} to find values of $C$ so that $\Gamma_{Cz^n}$ bounds a Jordan domain.  This question was mostly answered in \cite{FL17}, where Fleeman and Lundberg found precise values of $C$ that are sufficient to guarantee that $\Gamma_{Cz^n}$ bounds a Jordan domain. We will complete the result from \cite{FL17} by showing that their result is both necessary and sufficient and also by showing that for these particular values of $C$, there is exactly one Jordan domain bounded by the set $\Gamma_{Cz^n}$.

Our main results will generalize the result from \cite{FL17} in several ways.  First, we will consider the sets
\[
\Gamma_{n,j}(C,k):=\{z:\mbox{Re}[Cz^n+z^j]-|z|^2+k=0\},
\]
where we want to think of $j$ and $n$ as fixed.  Our goal is to find the relationship between $C$ and $k$ that is necessary and sufficient to ensure the existence of a unique Jordan domain whose boundary is part of the set $\Gamma_{n,j}(C,k)$.  Our results on this problem are presented in Section \ref{results} below. In some cases we are able to accomplish this goal exactly, while in the most general case we have only partial results.  In Section \ref{ex}, we work through the details of an example where the torsional rigidity of the bounded region can be calculated analytically.

\section{Results}\label{results}

We will be analyzing the sets $\Gamma_{n,j}(C,k)$ defined in the previous section.  For notational convenience, we also define the functions
\[
f_{n,j,C,k}(r,\theta)=Cr^n\cos(n\theta)+r^j\cos(j\theta)-r^2+k,
\]
so that
\[
\Gamma_{n,j}(C,k)=\{r\eitheta:f_{n,j,C,k}(r,\theta)=0\}.
\]
The following lemma will be useful in our later analysis.

\begin{lemma}\label{surround}
Consider $C\in\bbR$, $k>0$ and $j\in\{0,1,2\}$ and $n\in\bbN$ with $n>j$.  If the set
    \[
    \Gamma_{n,j}(C,k)=\{z:\textrm{Re}\left[Cz^n+z^j\right]-|z|^2+k=0\}
    \] 
bounds a Jordan domain $\Omega$, then $0\in\Omega$.
\end{lemma}

\begin{proof}
Let us think of the function $f_{n,j,C,k}(r,\theta)$ as a polynomial function of $r$ with $\theta$ fixed.  Since $k>0$, the hypotheses imply that the coefficients of $f_{n,j,C,k}(r,\theta)$ (listed in increasing order of degree) have at most two sign changes for any $\theta\in[0,2\pi]$.  It follows from Descartes' Rule of Signs that $f_{n,j,C,k}(r,\theta)$ has at most two zeros along any fixed ray emanating from the origin.

If $\Gamma_{n,j}(C,k)$ bounds a Jordan domain $\Omega$, then it must be the case that either $f_{n,j,C,k}(r,\theta)<0$ for all $r\eitheta\in\Omega$ or $f_{n,j,C,k}(r,\theta)>0$ for all $r\eitheta\in\Omega$.  The former possibility is easy to exclude. Indeed, such a Jordan domain would require $f$ to attain a local minimum somewhere in the complex plane. We note that $\Delta f_{n,j,C,k}(r,\theta)=-4$, so this is impossible.

Now we must exclude the possibility of a Jordan domain $\Omega$ that does not contain $0$, is bounded by $\Gamma_{n,j}(C,k)$,  and where $f_{n,j,C,k}(r,\theta)>0$ for all $r\eitheta\in\Omega$.  Suppose for contradiction that such an $\Omega$ exists.  Note that $f_{n,j,C,k}(0,\theta)=k>0$.  Let $\Omega_0$ be the connected component of $\{re^{i\theta}:f_{n,j,C,k}(r,\theta)>0\}$ that includes the origin.  If we start at the origin and travel along a ray that passes through $\Omega$, then that ray must leave $\Omega_0$, enter $\Omega$, and then leave $\Omega$ (since $\bar{\Omega}$ is compact).  This would require three zeros of $f_{n,j,C,k}(r,\theta)$ along that ray, which we have already seen is impossible.  Notice that this reasoning is still valid if $\Omega$ and $\Omega_0$ share a boundary point, since such a point would be a double zero of $f_{n,j,C,k}(r,\theta)$.
\end{proof}

We can actually strengthen the conclusion of Lemma \ref{surround} with the following proposition.

\begin{prop}\label{uniquesurr}
Given the same hypotheses as Lemma \ref{surround},  the set
    \[
    \Gamma_{n,j}(C,k)=\{z:\textrm{Re}\left[Cz^n+z^j\right]-|z|^2+k=0\}
    \] 
bounds at most one Jordan domain.
\end{prop}

\begin{proof}
    We consider the function $f_{n,j,C,k}(r,\theta)$ and observe that $f_{n,j,C,k}(0,\theta)>0$. Suppose $\Gamma_{n,j}(C,k)$ includes the boundary of a Jordan domain $\Omega$.  By Lemma \ref{surround}, we know that any Jordan domain bounded by $\Gamma_{n,j}(C,k)$ must include $0$ and we also saw in the proof of Lemma \ref{surround} that $\Gamma_{n,j}(C,k)$ cannot bound a region on which $f_{n,j,C,k}(r,\theta)<0$.  Thus, the only way $\Gamma_{n,j}(C,k)$ could bound a second Jordan domain is if there are two components of the set $\{r\eitheta:f_{n,j,C,k}(r,\theta)>0\}$ that share a common boundary curve (e.g. if $\Gamma_{n,j}(C,k)$ includes two concentric circles with $f_{n,j,C,k}$ positive between them).  Note that every point on this shared boundary curve would be a double zero of $f_{n,j,C,k}(r,\theta)$.  This would imply that for all those values of $\theta$, the function $f_{n,j,C,k}(r,\theta)$ has at least $3$ zeros, which contradicts Descartes' Rule of signs.
\end{proof}

The first set that we will consider is $\Gamma_{n,1}(C,k)$. In this case, we are able to determine the exact relationship between $C$ and $k$ for this set to bound a Jordan domain and we are able to show that this region is unique when it exists.

\begin{theorem}\label{gamman1}
    Suppose $C,k>0$ and $n\geq3$.  The set $\{z:\textrm{Re}[Cz^n+z]-|z|^2+k=0\}$ bounds a Jordan domain if and only if $C\leq C^*$ where
    \[
    C^*:=\frac{(2n-4)^n\left(4k(n-2)+\left((n-1)+\sqrt{(n-1)^2+4nk(n-2)}\right)\right)}{2(n-2)^2\left((n-1)+\sqrt{(n-1)^2+4nk(n-2)}\right)^n}
    \]
    If $C\in(0,C^*]$, then there is exactly one such Jordan domain.
\end{theorem}

\begin{proof}
    The existence portion of the proof closely resembles the proof from \cite{FL17}.
    

    Let us think of $\theta$ as fixed and $f_{n,j,C,k}(r,\theta)$ as a polynomial in $r$. Notice that $f_{n,j,C,k}(0,\theta)=k>0$ and for any $\theta\in\bbR$ and $r\in(0,\infty)$ it holds that 
    $f_{n,j,C,k}(r,\theta)\leq f_{n,j,C,k}(r,0)$.  
    Therefore, if there is a value $r_*\in(0,\infty)$ for which $f(r_*,0)<0$, then this will imply existence of a bounded connected component of $\bbC\setminus\Gamma_{n,1}(C,k)$. Such an $r_*$ exists if and only if
    \begin{equation}\label{cmax}
    0<C\leq\max_{r\in(0,\infty)}\frac{r^2-r-k}{r^n}.
   \end{equation}
   By considering large values of $r$, we see that the maximum in \eqref{cmax} must be positive.  To calculate this maximum, we take the derivative with respect to $r$ to obtain
\[
\frac{r^n(2r-1)-nr^{n-1}(r^2-r-k)}{r^{2n}}=\frac{r^2(2-n)+r(n-1)+nk}{r^{n+1}}
\]
This equals zero if and only if 
\[
r^2(2-n)+r(n-1)+nk=0
\]
The unique positive root of this polynomial is equal to 
\[
r_*=\frac{(n-1)+\sqrt{(n-1)^2+4nk(n-2)}}{2(n-2)}
\]
Substituting $r_*$ back into our upper bound for $C$ we find that $f(r,0)$ attains a negative value if and only if 
\begin{align*}
C&<\frac{r_*^2-r_*-k}{r_*^n}\\&=\frac{(2n-4)^n\left(4k(n-2)+\left((n-1)+\sqrt{(n-1)^2+4nk(n-2)}\right)\right)}{2(n-2)^2\left((n-1)+\sqrt{(n-1)^2+4nk(n-2)}\right)^n}\\&=C^*
\end{align*}
Thus whenever $0<C<C^*$, there must be at least one Jordan domain bounded by $\Gamma_{n,1}(C,k)$.  By Lemma \ref{surround}, we know that $0$ is in this Jordan domain and that if $C>C^*$, then $\Gamma_{n,j}(C,k)$ does not bound a Jordan domain.  All that remains is to prove the statement about uniqueness and this follows from Proposition \ref{uniquesurr}. 

\end{proof}

To complete our analysis of the sets $\Gamma_{n,1}(C,k)$, we now consider the cases $n=1$ and $n=2$. This is the content of the next two propositions.

\begin{prop}\label{n1}
The set
\[
\{z:\textrm{Re}[Cz]-|z|^2+k=0\}
\] 
consists of a single connected component with the following properties:
\begin{itemize}
\item[1)] For $C=0$ the set is a circle centered at the origin.
\item[2)] For $C\not=0$ the set is a circle not centered at the origin.
\end{itemize}
\end{prop}

\begin{proof}
    If $z=x+iy$, then $\textrm{Re}\left[Cz\right]=|z|^2-k$ may be equivalently written as 
    
    \[
    x^2+y^2-Cx=k.
    \]
 When $C=0$ this is the equation of a circle centered at $(0,0)$. If $C\not=0$, then this is the equation of a circle centered at $(C/2,0)$ and of radius $\sqrt{k+C^2/4}$.
\end{proof}

Notice that in the case $C\neq0$ of the previous theorem, the origin is inside the circle $\Gamma_{1,1}(C,k)$.

\begin{prop}\label{n2}
Suppose $k>0$.  The set $\Gamma_{2,1}(C,k)$ satisfies the following:
\begin{itemize}
\item[1)] For $C=0$ the set is a circle.
\item[2)] For $-1<C<1$ the set is an ellipse.
\item[3)] For $C=1$, then set is a parabola.
\item[4)] For $C=-1$, the set is a pair of vertical lines.
\item[5)] For $C<-1$ the set is a hyperbola.
\item[6)] For $C>1$ and $k\neq1/(4(C-1))$, the set is a hyperbola.
\item[7)] For $C>1$ and $k=1/(4(C-1))$, the set is a pair of lines.
\end{itemize}
Only in cases (1) and (2) does the set $\Gamma_{2,1}(C,k)$ bound a Jordan domain.
\end{prop}

\begin{proof}
    If $z=x+iy$, then $\textrm{Re}\left[Cz^2+z\right]=|z|^2-k$ may be equivalently written as 
    
    \[
    x^2(1-C)+y^2(1+C)-x=k
    \]
When $C=0$ this is the equation of a circle, and for all $C\in(-1,1)$ this is the equation of an ellipse. If $C=1$, then this becomes the equation of a parabola.  If $C=-1$, then the curve is a pair of vertical lines. If $C<-1$, then the equation can be rewritten
\[
(x\sqrt{1-C}-1/\sqrt{4-4C})^2+(1+C)y^2=k-1/(4C-4)
\]
from which we conclude the set is a hyperbola.  Cases (6) and (7) can be verified with similar calculations.
\end{proof}

We also include the following variant of Proposition \ref{n2}, in which the constant $C$ is multiplied by the function $z$ as well.

\begin{prop}\label{n2a}
Suppose $k>0$.  The set
\[
\{z:C\textrm{Re}[z^2+z]-|z|^2+k=0\}
\]
satisfies the following properties:
\begin{itemize}
\item[1)] For $C=0$ the set is a circle.
\item[2)] For $-1<C<1$ the set is an ellipse.
\item[3)] For $C=1$, then set is a parabola.
\item[4)] For $C=-1$, the set is a pair of vertical lines.
\item[5)] For $C<-1$ the set is a hyperbola.
\item[6)] For $C>1$ and $k\neq C^2/(4(C-1))$, the set is a hyperbola.
\item[7)] For $C>1$ and $k=C^2/(4(C-1))$, the set is a pair of lines.
\end{itemize}
Only in cases (1) and (2) does the set bound a Jordan domain.
\end{prop}

\begin{proof}
If $z=x+iy$, then $C\textrm{Re}\left[z^2+z\right]=|z|^2-k$ may be equivalently written as
\[
x^2+y^2-C(x^2-y^2+x)=k
\]
Each of the cases can be verified using calculations similar to those in the proof of Proposition \ref{n2}.
\end{proof}

Our next step is to consider the sets $\Gamma_{n,2}(C,k)$. In this case we are able to show that such sets never bound a Jordan domain.

\begin{theorem}
    If $C,k>0$, then the set $\Gamma_{n,2}(C,k)$ does not bound a Jordan domain.
\end{theorem}

\begin{proof}
    By Lemma \ref{surround}, it suffices to show that $\Gamma_{n,2}(C,k)$ does not bound a Jordan domain that contains the origin.  Notice that $f_{n,2,C,k}(r,0)=Cr^n+k$ has no positive real roots. This precludes the possibility of a Jordan domain that includes the origin whose boundary is in $\Gamma_{n,2}(C,k)$.
\end{proof}

Now we turn to the most general case for the sets $\Gamma_{n,j}(C,k)$ where $n>j>2$. Here we are able to determine the relationship between $C$, $k$, $n$, and $j$ for the existence of a Jordan domain that contains the origin. We may suppose $j>2$ since the cases $j=1$ and $j=2$ have already been addressed.

\begin{theorem}
    If $C,k>0$ and $n>j>2$ are natural numbers, then the set $\Gamma_{n,j}(C,k)$ bounds a Jordan domain that contains the origin if and only if $C<C^*$ and $k<k^*$ where
    \[
    C^*:=\max\limits_{r\in(0,\infty)}\frac{r^2-r^j-k}{r^n}
    \]
    and
    \[
    k^*:=\left(1-\frac{2}{j}\right)\left(\frac{2}{j}\right)^{2/(j-2)}.
    \]
\end{theorem}

\begin{proof}
As in our earlier proofs, it suffices to determine if and when $f_{n,j,C,k}(r,0)$ becomes negative somewhere in $(0,\infty)$.  We calculate
\begin{align*}
    f(r,0)>0 &\iff\\
    Cr^n+r^j-r^2+k>0 &\iff\\
    C>\frac{r^2-r^j-k}{r^n}
\end{align*}
Thus $f(r,0)$ attains a negative value in $(0,\infty)$ if and only if
\[
C<\max\limits_{r\in(0,\infty)}\frac{r^2-r^j-k}{r^n}.
\]
Let us set
\[
R(r)=\frac{r^2-r^j-k}{r^n}.
\]
There is no Jordan domain containing $0$ and bounded by $\Gamma_{n,j}(C,k)$ if $R(r)<0$ for all $r\in(0,\infty)$.  We can determine precisely when this happens.

Notice $R(r)>0$ whenever $r^2-r^j-k>0$, so it suffices to determine when the maximum value of $r^2-r^j-k$ is positive.  Taking the derivative and setting it equal to zero, we find that the maximum value of $r^2-r^j-k$ is attained when $r=r_*$ with
\[
r_*:=\left(\frac{2}{j}\right)^{1/(j-2)}
\]
We evaluate $R(r_*)$ to obtain
\[
R(r_*)=\frac{\left(\frac{2}{j}\right)^{2/(j-2)}-\left(\frac{2}{j}\right)^{j/(j-2)}-k}{\left(\frac{2}{j}\right)^{n/(j-2)}}
\]
We conclude that $R(r_*)>0$ if and only if
\[
k<\left(1-\frac{2}{j}\right)\left(\frac{2}{j}\right)^{2/(j-2)}=k^*.
\]



\end{proof}

For our final result of this section, we will strengthen a result of Fleeman and Lundberg from \cite{FL17}.  In Theorem \ref{lund+}, if we set $k=1$, then we recover the result from \cite{FL17}.  Our result is more general in that we allow for any $k>0$, provide necessary and sufficient conditions for the existence of a Jordan domain, and show that the Jordan domain - when it exists - is unique.

\begin{theorem}\label{lund+}
    For $k>0$ and $2<n\in\bbN$, the set $\{z:C\textrm{Re}[z^n]-|z|^2+k=0\}$ bounds a unique Jordan domain whenever 
    \[
    |C|\leq C^*
    \]
    where
    \[
    C^*:=\frac{2k}{n-2}\left(\frac{n-2}{nk}\right)^{n/2}.
    \]
If $|C|>C^*$, then the set $\{z:C\textrm{Re}[z^n]-|z|^2+k=0\}$ does not bound a Jordan domain.
\end{theorem}

\begin{proof}
Define
\[
\Lambda_n(C,k):=\{z:C\textrm{Re}[z^n]-|z|^2+k=0\}
\]
and
\[
f(r,\theta):=C\cos(n\theta)r^n-r^2+k
\]
Without loss of generality we may assume that $C>0$ by the periodicity of cosine. Notice $f(0,\theta)$ is positive and $f(r,\theta)\leq f(r,0)$, so by Lemma \ref{surround}, we may determine when $\Lambda_n(C,k)$ bounds a Jordan domain by determining if there is a value $r_*\in (0,\infty)$ for which $f(r_*,0)\leq 0$. Such a condition is satisfied if and only if
\begin{equation}
0<C\leq\max_{r\in(0,\infty)}\frac{r^2-k}{r^n}
\end{equation}
We may determine the maximum of the right-hand side of this inequality by taking its derivative with respect to $r$ as follows:
\begin{align*}
    \frac{d}{dr}\left[\frac{r^2-k}{r^n}\right]&=\frac{2r^{n+1}-(r^2-k)(nr^{n-1})}{r^{2n}}=\frac{r^2(2-n)+nk}{r^{n+1}}
\end{align*}
The rational function above attains a value of zero if and only if $r^2(2-n)+nk=0$. This zero occurs at the value
\[
r_*:=\sqrt{\frac{nk}{n-2}}
\]
(recall $n>2$).  We now evaluate our upper bound of equation (2) with $r=r_*$ to determine that $\Lambda_n(C,k)$ bounds a Jordan domain that contains $0$ precisely when 
\[
C\leq\frac{r_*^2-k}{r_*^n}=\frac{2k}{n-2}\left(\frac{n-2}{nk}\right)^{n/2}=C^*.
\]
The uniqueness follows from Lemma \ref{surround} and Proposition \ref{uniquesurr}.
\end{proof}

In the next section we inspect the unique Jordan domain that includes the origin whose boundary is part of the set $\Lambda_4(C,k)$ and calculate the torsional rigidity of such a region as $C$ and $k$ vary.

\section{An Example}\label{ex}

As an exercise one may consider calculating the torsional rigidity of the Jordan domain bounded by the set $\Gamma_{4,4}(C,k)$ as described in Section 1. Recall that
\[
\Gamma_{4,4}(C,k)=\{z:\textrm{Re}[Cz^4+z^4]-|z|^2+k=0\}
\]
We may let $\hat{C}$ be $C+1$ so we have 
\begin{align*}
\Gamma_{4,4}(C,k)&=\{z:\textrm{Re}[\hat{C}z^4]-|z|^2+k=0\}\\
&=\Lambda_4(\hat{C},k)
\end{align*}
which satisfies the conditions of Theorem \ref{lund+}.  This means that $\Lambda_4(\hat{C},k)$ bounds a Jordan domain if and only if $\hat{C}\leq k\left(\frac{2}{4k}\right)^2=\frac{1}{4k}$. In this case, the boundary of that Jordan domain (call it $\Omega$) may be parameterized (in polar coordinates) by
\[
\left\{\left(\alpha,\theta\right),0\leq\theta\leq 2\pi\right\}
\]
where
\[
\alpha:=\left(\frac{1-\sqrt{1-4\hat{C}k\cos(4\theta)}}{2\hat{C}\cos(4\theta)}\right)^{1/2}.
\]

We can use this parametrization to calculate the torsional rigidity of this Jordan domain by evaluating the integral
\begin{equation}\label{rhoexp}
\int_\Omega|\bar{z}-f(z)|^2dA(z).
\end{equation}
where $f(z)=\frac{d}{dz}\left[F(z)\right]$ and $F(z)=\frac{\hat{C}z^4+k}{2}$. 

\begin{figure}
    \centering
    \includegraphics[scale=0.8]{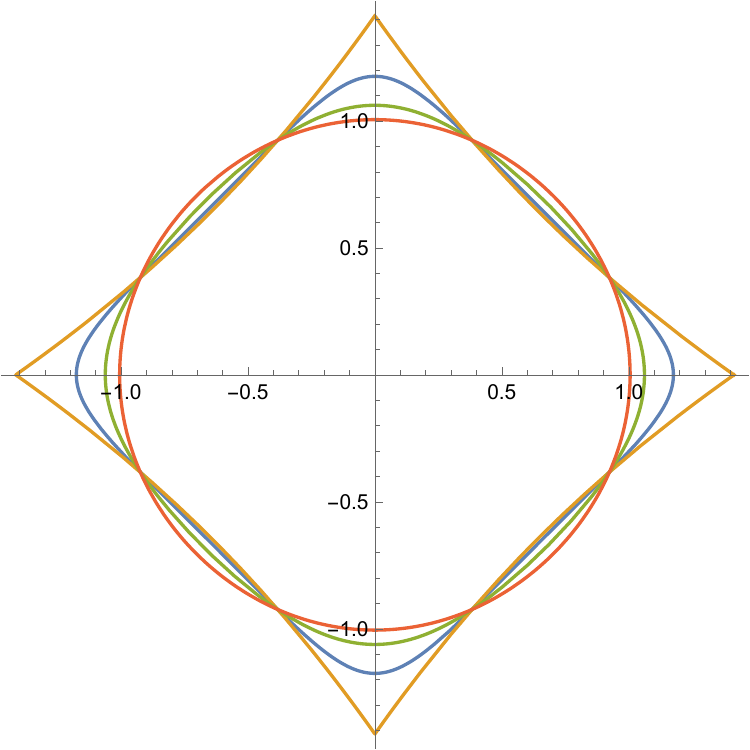}
    \caption{\textit{A plot of the set $\{(\alpha,\theta),0\leq\theta\leq 2\pi\}$ with $k=1$ for four different values of $\hat{C}$, ranging from $\hat{C}=\frac{1}{4k}$ at the outermost orange curve, $\hat{C}=\frac{1}{5k}$ in blue, $\hat{C}=\frac{1}{10k}$ in green, and $\hat{C}=\frac{1}{100k}$ as the nearly circular curve in red.}}
    \label{label5}
\end{figure}

When $\hat{C}=\frac{1}{4k},\frac{1}{5k},\frac{1}{10k},$ and $\frac{1}{100k}$, the integral evaluates to 1.63988$k^2$, 1.60815$k^2$, 1.57894$k^2$, and 1.57087$k^2$ respectively (rounding to 5 decimal places). Notice that this appears to be approaching $k^2\pi/2$, the torsional rigidity of a disk whose radius is $\sqrt{k}$, as $\hat{C}k$ approaches zero.  This is exactly what one would expect given that if we fix $k=1$ and send $\hat{C}\rightarrow0$, the Jordan domain bounded by $\Gamma_{4,4}(C,k)$ converges to the unit disk (see Figure 1).  We also see that for any fixed value of $\hat{C}k$, it holds that $\rho(\Omega)\rightarrow\infty$ quadratically as $k\rightarrow\infty$. 

When calculating torsional rigidity of Jordan domains, extremal problems often arise, which ask one to maximize the torsional rigidity of a region under various constraints.  We can consider such a problem in this setting by asking what value of $\hat{C}\in[0,(4k)^{-1}]$ maximizes the torsional rigidity of the Jordan domain bounded by $\Gamma_{4,4}(C,k)$ (we think of $k$ as a fixed value in this problem).  Numerical calculations show that torsional rigidity decreases as $\hat{C}\rightarrow 0$ (see Figure \ref{label6}).  It appears that the value of $\hat{C}$ that produces the Jordan domain with maximal torsional rigidity having boundary in the set $\Gamma_{4,4}(C,k)$ is the one where $\hat{C}$ attains its upper bound, $\hat{C}=\frac{1}{4k}$.  While we do not have a closed form expression for the torsional rigidity of the Jordan domain bounded by $\Gamma_{4,4}(C,k)$ for every $C$ and $k$, we recall that the integral \eqref{rhoexp} involves simple enough expressions that we have high confidence that the plot in Figure \ref{label6} gives an accurate description of what is true.
\begin{figure}
    \centering
    \includegraphics[scale=0.8]{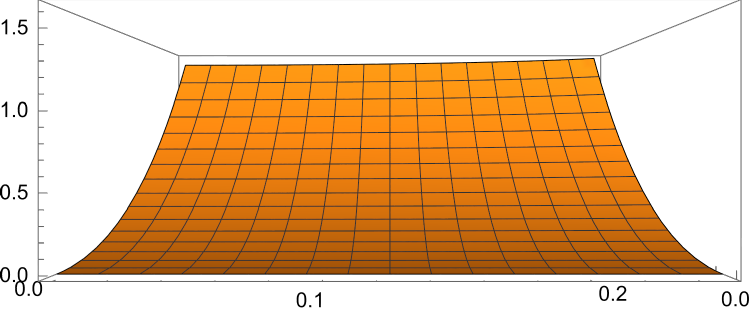}
    \caption{\textit{A 3-dimensional plot of the torsional rigidity $\int_\Omega|\bar{z}-f(z)|^2dA(z)$ with $\Omega$ the Jordan domain parametrized above. $\hat{C}$ ranges from $0$ to $\frac{1}{4k}$ and $k$ ranges from $0$ to $1$.}}
    \label{label6}
\end{figure}


\bigskip

\noindent\textbf{Acknowledgements.}  The authors are grateful for the input of Dmitry Khavinson, who encouraged us to pursue this line of research.  The second author also gratefully acknowledges support from the Simons Foundation through collaboration grant 707882.

\end{document}